\documentclass[a4paper,10pt]{amsart}
\usepackage[utf8]{inputenc}
\usepackage{amsxtra}
\usepackage{amsopn}
\usepackage{amsmath,amsthm,amssymb}
\usepackage{amscd}
\usepackage{amsfonts}
\usepackage{latexsym}
\usepackage{verbatim}
\usepackage[all]{xy}
\usepackage{color}

\DeclareMathOperator{\Ker}{Ker}
\theoremstyle{plain}
\newtheorem{theorem}{Theorem}[section]
\newtheorem*{theorem*}{Theorem}
\newtheorem{definition}[theorem]{Definition}

\newtheorem{prop}[theorem]{Proposition}
\newtheorem{cor}[theorem]{Corollary}
\newtheorem{rem}[theorem]{Remark}
\newtheorem{ex}[theorem]{Example}
\newtheorem*{mt*}{Main Theorem}


\newcommand\R{{\mathbb R}}
\newcommand\Z{{\mathbb Z}}

\newcommand{\del}{\partial}
\newcommand{\delbar}{\overline{\del}}

\DeclareMathOperator{\Imm}{Im}

\title[Symplectic cohomologies and deformations]{Symplectic cohomologies and deformations}

\author[Nicoletta Tardini]{Nicoletta Tardini}
\address[Nicoletta Tardini]{Dipartimento di Matematica\\
Universit\`a di Pisa\\
largo Bruno Pontecorvo 5\\
56127 Pisa, Italy
}
\email{nicoletta.tardini@gmail.com}
\email{tardini@mail.dm.unipi.it}

\author[Adriano Tomassini]{Adriano Tomassini}
\address[Adriano Tomassini]{Dipartimento di Scienze Matematiche, Fisiche e Informatiche\\
Plesso Matematico e Informatico\\
Università di Parma\\
Parco Area delle Scienze 53/A\\
43124 Parma, Italy
}
\email{adriano.tomassini@unipr.it}

\thanks{\newline The first-named author is supported by SIR2014 project RBSI14DYEB ``Analytic aspects in complex and hypercomplex geometry'' and by GNSAGA of INdAM.
The second-named author is supported by Project PRIN ``Varietà reali e complesse: geometria, topologia e analisi armonica'' and by GNSAGA of INdAM}
\begin{document}

\maketitle

\begin{abstract}
In this note we study the behavior of symplectic cohomology groups under symplectic deformations. Moreover, we show that for compact almost-K\"ahler manifolds $(X,J,g,\omega)$ with $J$ $\mathcal{C}^\infty$-pure and full the space of de Rham harmonic forms is contained in the space of symplectic-Bott-Chern harmonic forms. 
Furthermore, we prove that the second non-HLC degree measures the gap between the de Rham and the symplectic-Bott-Chern harmonic forms.
\end{abstract}

\section{Introduction}

A very special class of smooth manifolds is represented by \emph{K\"ahler manifolds}; these are indeed endowed with three geometric structures which are compatible to each other. More precisely, K\"ahler manifolds have a \emph{complex} structure, a \emph{metric} structure and a \emph{symplectic} structure which are related. This relation leads to important results for K\"ahler manifolds, even at a topological level. For instance, there are topological obstructions to the existence of K\"ahler metrics on compact manifolds: the odd Betti numbers are even and the even Betti numbers are positive. Cohomology groups are global tools in studying smooth manifolds and it is interesting to notice that when we weaken the K\"ahler assumption new cohomology groups occur which are not isomorphic to the classical ones. In particular, these cohomology groups provide additional informations on non-K\"ahler manifolds.\\
In this note, we deal with symplectic manifolds which are not K\"ahler in general. Let $(X^{2n},\omega)$ be a symplectic manifold of dimension $2n$, then L.S. Tseng and S.-T. Yau in \cite{TY} introduce the symplectic analogue of the complex Bott-Chern and Aeppli cohomology groups. Namely, denoting with $d^\Lambda$ the \emph{symplectic co-differential} they define the {\em symplectic Bott-Chern} and the {\em symplectic Aeppli cohomology groups} respectively as
\[
H^\bullet_{d+d^\Lambda}\left(X\right)
:=\frac{\ker(d+d^\Lambda)}{\Imm dd^\Lambda}\,,
\qquad
H^\bullet_{dd^\Lambda}\left(X\right)
:=\frac{\ker(dd^\Lambda)}{\left(\Imm d+\Imm d^\Lambda\right)}.
\]
If $X$ is compact, by Hodge theory these cohomology groups are finite-dimensional; more precisely, they are isomorphic to the kernel of suitable
$4^{\text{th}}$-order elliptic self-adjoint differential operators. These cohomology groups have been introduced since the de Rham cohomology is not the appropriate cohomology to talk about symplectic-Hodge theory. Indeed, the existence of a $d$-closed and $d^\Lambda$-closed representative in every de Rham cohomology class turns out to be equivalent to the \emph{Hard-Lefschetz condition} (cf. \cite{brylinski}, \cite{mathieu}, \cite{merkulov},
\cite{yan}, \cite{cavalcanti-phd}), i.e., the maps
\[
L^k=[\omega]^k:H^{n-k}_{dR}(X)\longrightarrow H^{n+k}_{dR}(X), \qquad 0\leq k\leq n
\]
are isomorpisms.
Moreover, this is also equivalent to the \emph{$dd^\Lambda$-lemma}, namely all the natural maps induced by the identity in the diagram
$$ \xymatrix{
  & H^{\bullet}_{d+d^\Lambda}(X) \ar[ld]\ar[rd] & \\
  H^{\bullet}_{dR}(X) \ar[rd] &  & H^{\bullet}_{d^\Lambda}(X) \ar[ld] \\
  & {\phantom{\;.}} H^{\bullet}_{dd^\Lambda}(X) \; &
} $$
are isomorphisms (where $H^\bullet_{d^\Lambda}\left(X\right)
:=\frac{\ker(d^\Lambda)}{\Imm d^\Lambda}$).\\
In dimension $2n=4$ a quantitative characterization of the Hard-Lefschetz condition has been proven in \cite{tardini-tomassini-symplectic} in terms of the vanishing of a natural number called $\Delta^2_s$.
In particular, if $\omega$ is the fundamental form of a K\"ahler metric then all those cohomology groups are isomorphic. Therefore they provide additional informations on non-K\"ahler symplectic manifolds.\\
Notice that these cohomology groups, by definition, are invariant under symplectomorphisms. In this note we study the behavior of the dimension of these cohomology groups under symplectic deformations. According to \cite{debartolomeis}, roughly speaking $H^2_{dR}(X;\mathbb{R})$
represents the tangent space of the moduli space of symplectic deformations and the theory is totally unobstructed. Once shown in Proposition \ref{symplectic-semicontinuity} that if we
consider a small deformation
of symplectic structures $\left\lbrace\omega_t\right\rbrace_t$ on $X$ such that
$\omega_0=\omega$
then the functions
$$
t\mapsto \dim H^\bullet_{d+d^{\Lambda_t}}
$$
are upper-semicontinuous functions of $t$, we construct an explicit example where the dimension drops.
Another example with a similar aim is constructed taking into account other symplectic cohomologies. For any $r,s\in\mathbb{N}$, we consider
$$
H^{(r,s)}_\omega(X\,;\,\mathbb{R})\,:=\,
\left\lbrace[L^r\beta^{(s)}]\in H^{2r+s}_{dR}(X\,;\,\mathbb{R})\,:\,
\beta^{(s)}\in \ker\Lambda\cap\Omega^s(X)\right\rbrace\subseteq H^{2r+s}_{dR}(X\,;\,\mathbb{R})\,.
$$
By \cite[Corollary 2.5]{angella-tomassini-symplectic} if $(X,\omega)$ satisfies the Hard-Lefschetz condition then the Lefschetz decomposition on differential forms passes in cohomology, that is
\begin{equation}\label{omega-decomposition}
H^\bullet_{dR}(X\,;\,\mathbb{R})=\bigoplus_{r\in\mathbb{N}}
H^{(r,\bullet-2r)}_\omega(X\,;\,\mathbb{R})\,.
\end{equation}
On degree $2$ by \cite[Theorem 2.6]{angella-tomassini-symplectic} the decomposition is always true (on any dimension) leading therefore to the symplectic analogue of the T. Dr\v{a}ghici, T.-J. Li and W. Zhang's \cite[Theorem 2.3]{DLZ} in the complex setting (true only on real dimension $4$).
With an explicit example we show that the decomposition in higher degree is not closed under symplectic deformations. In particular, we will see that the dimensions of the cohomology groups $H^{(r,s)}_\omega(X\,;\,\mathbb{R})$ change along a deformation.\\
We can resume the results in the following
\begin{prop}
Let $X:=\Gamma\backslash G$ be the $6$-dimensional nilmanifold with structure equations
$$
\left(0,0,0,12,14,15+23+24\right),
$$
and let $\omega$ be the invariant symplectic structure on $X$ defined by
$\omega=e^{16}+e^{25}-e^{34}$. Then there exists a symplectic deformation
$\left\lbrace\omega_t\right\rbrace_t$ such that
\begin{itemize}
\item[1.] the dimensions of $H^\bullet_{d+d^{\Lambda_t}}\left(X\right)$ drop varying $t$;
\item[2.] the dimensions of $H^{(r,s)}_{\omega_t}(X\,;\,\mathbb{R})$ jump varying $t$;
\item[3.] if $t\neq 0$ we have the following decomposition
$$
H^{(0,3)}_{\omega_t}(X\,;\mathbb{R})\oplus
H^{(1,1)}_{\omega_t}(X\,;\mathbb{R})\,=\,
H^3_{dR}(X\,;\,\mathbb{R})\,.
$$
However, for $t=0$ we have
$$
H^{(0,3)}_{\omega}(X\,;\mathbb{R})+
H^{(1,1)}_{\omega}(X\,;\mathbb{R})\subsetneq
H^3_{dR}(X\,;\,\mathbb{R})
$$
and
$$
H^{(0,3)}_{\omega}(X\,;\mathbb{R})\cap
H^{(1,1)}_{\omega}(X\,;\mathbb{R})\neq \left\lbrace 0\right\rbrace\,.
$$
\end{itemize}
In particular, if $X$ does not satisfy the Hard-Lefschetz Condition then the decomposition (\ref{omega-decomposition}) is not closed under symplectic deformations.
\end{prop}
In the second part of this note we show that the de Rham cohomology and the symplectic Bott-Chern cohomology are strictly related on compact almost-K\"ahler manifolds $(X^{2n},J,g,\omega)$ with $J$ $\mathcal{C}^\infty$-pure and full; more precisely one has the inclusion
$$
\mathcal{H}^2_{dR}(X)\subseteq \mathcal{H}^2_{d+d^\Lambda}(X)
$$
of the spaces of harmonic forms in degree $2$. In particular, on compact almost-K\"ahler $4$-manifold one gets the following (cf. Corollary \ref{cor:delta2-dimensionofv})
\begin{cor}
Let $(X^4,\omega)$ be a compact symplectic $4$-manifold and let $g$ be a compatible Riemannian metric. Then, the non-HLC degrees are
$$
\Delta_s^1=0 \quad\text{and}\quad
\Delta_s^2=\dim\left(\mathcal{H}^2_{d+d^\Lambda}(X)\cap \Imm d\right)\,.
$$
Moreover,
$\Delta_s^2=0$ if and only if the Hard-Lefschetz Condition holds on $(X^4\,,\omega)$.
\end{cor}
Explicit examples are also computed. In particular, we describe an explicit example of an almost-K\"ahler $6$-dimensional manifold $X$ with $J$ non $\mathcal{C}^\infty$-full but such that
$\mathcal{H}^2_{dR}(X)\subseteq \mathcal{H}^2_{d+d^\lambda}(X)$. Moreover,
$\mathcal{H}^3_{dR}(X)\not\subseteq\mathcal{H}^3_{d+d^\lambda}(X)$.

\section{Preliminaries}\label{preliminaries}

Let $X$ be a compact smooth manifold of even dimension $2n$ and denote by $\Omega^k(X)$ the space of smooth $k$-forms on $X$. Suppose that $X$ admits a symplectic structure $\omega$;
then many cohomology groups can be defined on $(X,\omega)$. In particular,
Tseng and Yau in \cite{TY}, noticing that the de Rham cohomology is not the appropriate cohomology to talk about symplectic Hodge theory, define a symplectic version of the Bott-Chern and 
the Aeppli cohomology groups (cf. \cite{aeppli}, \cite{bott-chern}). We denote with
$\star:\Omega^\bullet(X)\longrightarrow \Omega^{2n-\bullet }(X)$ 
the {\em symplectic-$\star$-Hodge operator} (see \cite{brylinski}) which is defined, for every $k\in \mathbb{N}$ and for every $\alpha,\beta\in\Omega^k(X)$, by
$$
\alpha\wedge\star\beta=(\omega^{-1})^k(\alpha,\beta)\frac{\omega^n}{n!}
$$
and let
$\Lambda:\Omega^\bullet(X)\longrightarrow \Omega^{\bullet -2}(X)$ be the adjoint of the
Lefschetz operator $L:=\omega\wedge-:\Omega^\bullet(X)\longrightarrow \Omega^{\bullet +2}(X)$; then
the {\em symplectic co-differential} is defined as
\[
d^\Lambda:=\left[d,\Lambda\right]=d\Lambda-\Lambda d
\]
and it turns out that the operator $d^\Lambda$ is the symplectic adjoint of the exterior derivative $d$, that is
$$
d^\Lambda_{\mid \Omega^k(X)}\,=\,(-1)^{k+1}\star d\star_{\mid \Omega^k(X)}.
$$
Since $\star^2\,=\,\text{Id}$, then $(d^\Lambda)^2\,=\,0$. 
In particular, it is natural to define (see \cite{TY})
$$
H^\bullet_{d^\Lambda}(X)\,:=\,\frac{\ker d^\Lambda}{\Imm d^\Lambda}\,.
$$
By the above considerations, it seems that $d^\Lambda$ is the symplectic counterpart of the Riemannian co-differential $d^*$. For this reason, J.-L. Brylinski in \cite{brylinski} defined a \emph{symplectic-harmonic form} to be a both $d$-closed and $d^\Lambda$-closed differential form, conjecturing that
there exists a (possibly non-unique) symplectic harmonic representative in every de Rham cohomology class (\cite[Conjecture 2.2.7]{brylinski}).
However, even if $d^\Lambda$ is the symplectic adjoint of $d$, notice that $dd^\Lambda+d^\Lambda d$ is not an elliptic operator, indeed
$dd^\Lambda+d^\Lambda d=0$.
This motivated the introduction of the {\em symplectic Bott-Chern cohomology groups} and the {\em symplectic Aeppli cohomology groups} defined respectively as (see \cite{TY})
\[
H^k_{d+d^\Lambda}\left(X\right)
:=\frac{\ker(d+d^\Lambda)\cap \Omega^k(X)}{\Imm dd^\Lambda\cap \Omega^k(X)}
\]
and 
\[
H^k_{dd^\Lambda}\left(X\right)
:=\frac{\ker(dd^\Lambda)\cap \Omega^k(X)}{\left(\Imm d+\Imm d^\Lambda\right)\cap \Omega^k(X)}.
\]
If we consider a compatible triple $(\omega,J,g)$ on $X$
(meaning that the almost-complex structure $J$ is $\omega$-calibrated and $g$ is the corresponding 
Riemannian metric on $X$) then, denoting with $*$ the standard \emph{Hodge-operator}
with respect to the Riemannian metric $g$,
there are canonical isomorphisms (see \cite{TY})
\[
\mathcal{H}^k_{d^\Lambda}\left(X\right):=\ker\Delta_{d^\Lambda}
\simeq H^k_{d^\Lambda}\left(X\right),
\]
where $\Delta_{d^{\Lambda}} :=  d^{\Lambda*}d^\Lambda+d^\Lambda d^{\Lambda*}$ is a second-order elliptic self-adjoint
differential operator and
\[
\mathcal{H}^k_{d+d^\Lambda}\left(X\right):=\ker\Delta_{d+d^\Lambda}
\simeq H^k_{d+d^\Lambda}\left(X\right),\qquad
\mathcal{H}^k_{dd^\Lambda}\left(X\right):=\ker\Delta_{dd^\Lambda}
\simeq H^k_{dd^\Lambda}\left(X\right).
\]
where
$\Delta_{d+d^{\Lambda}}$, $\Delta_{dd^{\Lambda}}$ are fourth-order elliptic self-adjoint
differential operators defined by
\[
\begin{array}{lcl}
\Delta_{d+d^{\Lambda}}& := &(dd^{\Lambda})(dd{^\Lambda})^*+(dd^{\Lambda})^*(dd^{\Lambda})+
d^*d^{\Lambda} d^{\Lambda *}d+d^{\Lambda *}d d^*d^{\Lambda}+d^*d+d^{\Lambda *}d^{\Lambda},\\[10pt]
\Delta_{dd^{\Lambda}} & :=& (dd^{\Lambda})(dd{^\Lambda})^*+(dd^{\Lambda})^*(dd^{\Lambda})+
dd^{\Lambda *}d^\Lambda d^*+d^\Lambda d^*dd^{\Lambda *}+dd^*
+d^\Lambda d^{\Lambda *}.
\end{array}
\]
In particular, the symplectic cohomology groups are finite-dimensional vector spaces on a compact
symplectic manifold. For
$\sharp\in\left\{d^\Lambda,d+d^\Lambda,
dd^\Lambda\right\}$ we denote $h^\bullet_\sharp:=:h^\bullet_\sharp(X):= \dim H^\bullet_\sharp(X)<\infty$ when the manifold $X$ is understood.\\
By definition, the identity induces natural maps as follows
$$ \xymatrix{
  & H^{\bullet}_{d+d^\Lambda}(X) \ar[ld]\ar[rd] & \\
  H^{\bullet}_{dR}(X,\mathbb{R}) \ar[rd] &  & H^{\bullet}_{d^\Lambda}(X). \ar[ld] \\
  & {\phantom{\;.}} H^{\bullet}_{dd^\Lambda}(X) \; &
} $$
In general, these maps are neither injective nor surjective; in particular
a compact symplectic manifold $(X,\omega)$ is said to satisfy the \emph{$dd^\Lambda$-lemma} (see \cite[Definition 3.12]{TY}) if the natural map
$H^\bullet_{d+d^\Lambda}(X)\to H^{\bullet}_{dR}(X,\mathbb{R})$ is injective, i.e.,
every $d^\Lambda$-closed, $d$-exact form
is also $dd^\Lambda$-exact.\\
By \cite[Lemma 5.15]{deligne-griffiths-morgan-sullivan} this is equivalent to say that all the maps in the above diagram are isomorphisms. In fact, this property is related to the Hard-Lefschetz condition and to the Brylinski conjecture. We resume all the known results just recalled in the following
(cf. \cite[Conjecture 2.2.7]{brylinski}, \cite[Corollary 2]{mathieu}, \cite[Proposition 1.4]{merkulov},
\cite[Theorem 0.1]{yan}, \cite[Theorem 5.4]{cavalcanti-phd}, 
\cite[Proposition 3.13]{TY} \cite[Theorem 4.4]{angella-tomassini-algebraic}).
\begin{theorem}
Let $(X,\omega)$ be a compact symplectic manifold. Then the following facts are equivalent:
\begin{itemize}
\item{} the \emph{Hard-Lefschetz condition} (\emph{HLC} for short) holds, i.e., for every $k\in\mathbb{Z}$, the maps
$$
L^k:H^{n-k}_{dR}(X,\mathbb{R})\to H^{n+k}_{dR}(X,\mathbb{R})
$$
are isomorphisms;
\item{} the \emph{Brylinski conjecture} holds, i.e., there exists a
symplectic harmonic form in each de Rham cohomology class;
\item{} the \emph{$dd^\Lambda$-lemma} holds;
\item{} the natural maps induced by the identity $H^{\bullet}_{d+d^\Lambda}(X)
\longrightarrow H^{\bullet}_{dR}(X,\mathbb{R})$ are injective;
\item{} the natural maps induced by the identity $H^{\bullet}_{d+d^\Lambda}(X)
\longrightarrow H^{\bullet}_{dR}(X,\mathbb{R})$ are isomorphisms;
\item{} the natural maps induced by the identity in the following diagram are isomorphisms
$$ \xymatrix{
  & H^{\bullet}_{d+d^\Lambda}(X) \ar[ld]\ar[rd] & \\
  H^{\bullet}_{dR}(X,\mathbb{R}) \ar[rd] &  & H^{\bullet}_{d^\Lambda}(X); \ar[ld] \\
  & {\phantom{\;.}} H^{\bullet}_{dd^\Lambda}(X) \; &
} $$
\item{} for every $k\in\mathbb{N}$, $h^k_{d+d^\Lambda}(X)+h^k_{dd^\Lambda}(X)=2b_k(X)$.
\end{itemize}
\end{theorem}
The last point in the list shows that there is a quantitative characterization of the $dd^\Lambda$-lemma in terms of the dimensions of the symplectic cohomology groups.
More precisely, in \cite{angella-tomassini-algebraic} D. Angella and the second-named author, starting from a purely
algebraic point of view,
introduce on a compact symplectic manifold $(X^{2n},\omega)$
the following non-negative integers 
\[
\Delta^k_s(X)\,:=\, h^k_{d+d^\Lambda}(X)+h^k_{dd^\Lambda}(X)-2b_k(X)\geq 0,
\qquad k\in\mathbb{Z},
\]
proving that, similarly to the complex case, their vanishing characterizes the $dd^\Lambda$-lemma.
As already observed in \cite{tardini-tomassini-symplectic} (see also \cite{chansuen}) we can write the non-HLC degrees
as follows
\[
\Delta^k_s\,=\,2(h^k_{d+d^\Lambda}-b_k),\qquad k\in\mathbb{Z};
\]
Therefore, one has that
for all $k=1,\ldots,n$
\[
b_k\,\leq\, h^k_{d+d^\Lambda}
\]
on a compact symplectic $2n$-dimensional manifold.\\
Moreover the equalities 
\[
b_k\,=\,h^k_{d+d^\Lambda}, \qquad k\in\mathbb{Z},
\]
hold on a compact symplectic $2n$-dimensional manifold if and only if it 
satisfies the Hard-Lefschetz condition.\\
In \cite{tardini-tomassini-symplectic} the authors, studying the similarities and differences with the analogue numbers in complex geometry, show that on any
compact symplectic manifold $\Delta^1_s=0$. In particular
on any compact symplectic $4$-manifold $(X^4\,,\omega)$
\[
HLC \iff \Delta_s^2=0 \iff b_2(X)\,=\,h^2_{d+d^\Lambda}(X).
\]
\begin{rem}
In complex geometry, one can define similarly the natural numbers $\Delta^k$ as in \cite{angella-tomassini-inventiones} and by \cite{teleman} (see also \cite{angella-tomassini-verbitsky}) one has that on a compact complex surface
$\Delta^1=0$ and $\Delta^2\in\left\lbrace 0,2\right\rbrace$. In fact, a similar result does not hold in symplectic geometry. Indeed, in \cite{tardini-tomassini-symplectic} the authors construct an explicit example of a compact symplectic $4$-manifold with $\Delta^2_s=4$. For a more detailed comparison between the complex and symplectic settings see \cite{tardini-proceeding}.
\end{rem}

\begin{rem}
Most of the results holding for the Tseng and Yau's symplectic cohomology groups
can be generalized to the locally conformally symplectic (\emph{lcs} for short) setting as done
by D. Angella, A. Otiman and the first-named author in \cite{angella-otiman-tardini}. A new notion of lcs-Hard-Lefschetz condition is introduced which forces the lcs structure to be globally conformally symplectic.
\end{rem}

Other cohomology groups can be considered on a compact symplectic manifold
$(X,\omega)$ which can give interesting decompositions for the de Rham cohomology.
More precisely, denoting with $P^\bullet(X):=\Ker\Lambda$ the space of \emph{primitive forms}, one can define the following spaces. For any $r,s\in\mathbb{N}$,
$$
H^{(r,s)}_\omega(X\,;\,\mathbb{R})\,:=\,
\left\lbrace[L^r\beta^{(s)}]\in H^{2r+s}_{dR}(X\,;\,\mathbb{R})\,:\,
\beta^{(s)}\in P^s(X)\right\rbrace\subseteq H^{2r+s}_{dR}(X\,;\,\mathbb{R});
$$
In general,
$$
\sum_{2r+s=k}H^{(r,s)}_\omega(X\,;\,\mathbb{R})\subseteq H^k_{dR}(X,\mathbb{R})
$$
but the sum is neither direct nor equal to the de Rham cohomology. 
Namely, in general the Lefschetz decomposition on differential forms \cite[Corollary 2.6]{yan}
$$
\Omega^{\bullet}(X)\,=\,\bigoplus_{r\in\mathbb{N}}L^rP^{\bullet-2r}(X)
$$
does not pass in cohomology.
However,
by \cite[Corollary 2.5]{angella-tomassini-symplectic} if $(X,\omega)$ satisfies the Hard-Lefschetz condition then
$$
H^\bullet_{dR}(X\,;\,\mathbb{R})=\bigoplus_{r\in\mathbb{N}}
H^{(r,\bullet-2r)}_\omega(X\,;\,\mathbb{R})\,.
$$
The total degree two is very special, indeed the decomposition
$$
H^2_{dR}(X\,;\,\mathbb{R})=H^{(1,0)}_\omega(X\,;\,\mathbb{R})
\oplus H^{(0,2)}_\omega(X\,;\,\mathbb{R})
$$
holds on every compact symplectic manifold by \cite[Theorem 2.6]{angella-tomassini-symplectic}. Notice that this decomposition can be considered as the symplectic analogue of the T. Dr\v{a}ghici, T.-J. Li and W. Zhang's \cite[Theorem 2.3]{DLZ} in the complex setting.

\section{Bott-Chern numbers and deformations}

Let $X$ be a smooth manifold and let $\eta$ be a closed $k$-form on $X$. A {\em deformation} of 
$\eta$ is a smooth family of closed $k$-forms $\eta_t$, with $t\in\R$, $\vert t\vert <\varepsilon$ and $\eta_0=\eta$. If $\Phi_t$ is a diffeomorphism flow 
on $X$, then $\eta_t=\Phi^*_t\eta$ is said to be a {\em non essential deformation} of $\eta$. A closed $k$-form $\theta$ is said to be an 
{\em infinitesimal deformation} (respectively {\em non essential infinitesimal deformation}) if there exists a deformation (respectively a non essential 
deformation) $\eta_t$ of $\eta$ such that $\frac{d}{dt}\eta_t\vert_{t=0}=\theta$. It is immediate to see that the space of infinitesimal deformations of 
$\eta$ can be identified with the space of closed $k$-forms on $X$. 

Let $(X,\omega)$ be a $2n$-dimensional closed symplectic manifold, let $\theta$ be a closed $2$-form on $X$ and set $\omega_t=\omega+t\theta$. Then in view of Moser's Lemma, 
$\omega_t=\Phi^*_t\omega$, for a $1$-parameter family of diffeomorphisms $\Phi_t$ of $X$, with $\Phi_0=\hbox{id}_X$ if and only if $\theta=d\beta$, namely the 
tangent space at $\omega$ to the space of germs of deformations 
of the symplectic form $\omega$ can be identified with with $H^2_{dR}(X;\R)$.

Let $(X,\omega)$ be a compact symplectic manifold and consider a small deformation
of symplectic structures $\left\lbrace\omega_t\right\rbrace_t$ on $X$ such that
$\omega_0=\omega$. According to this deformation
we obtain a curve of compatible almost-complex
structures $\left\lbrace J_t\right\rbrace_t$ on $X$. Therefore we can consider the
symplectic Bott-Chern and Aeppli Laplacian operators $\Delta_{d+d^{\Lambda_t}}$
and $\Delta_{dd^{\Lambda_t}}$ on $(X,\omega_t)$ as depending on the parameter $t$.
By the classical theory of elliptic operators we obtain that 
the functions $t\mapsto \dim\Ker\Delta_{d+d^{\Lambda_t}}$
and $t\mapsto \dim\Ker\Delta_{dd^{\Lambda_t}}$ are
upper-semicontinuous functions of $t$. Therefore we have the following
\begin{prop}\label{symplectic-semicontinuity}
Let $(X,\omega)$ be a compact symplectic manifold and consider a small deformation
of symplectic structures $\left\lbrace\omega_t\right\rbrace_t$ on $X$ such that
$\omega_0=\omega$.
The functions
$$t\mapsto h^\bullet_{d+d^{\Lambda_t}}
\qquad
t\mapsto h^\bullet_{dd^{\Lambda_t}}
$$
are upper-semicontinuous functions of $t$.
\end{prop}

Now we prove the following 
\footnote{Part of these computations have been performed by M. Rinaldi in her Master thesis, Propriet\`a Coomologiche di Variet\`a Simplettiche, advisor Prof. A. Tomassini, Universit\`a di Parma, 2014, \texttt{http://www.bibliomath.unipr.it/TesiDipNuovo.html}; we thank her.}
\begin{prop}
Let $X:=\Gamma\backslash G$ be the $6$-dimensional nilmanifold with structure equations
$$
\left(0,0,0,12,14,15+23+24\right),
$$
and let $\omega$ be the invariant symplectic structure on $X$ defined by
$\omega:=e^{16}+e^{25}-e^{34}$. Then there exists a symplectic deformation
$\left\lbrace\omega_t\right\rbrace_t$ such that
\begin{itemize}
\item[1.] $h^\bullet_{d+d^{\Lambda_t}}\left(X\right)$ drop varying $t$;
\item[2.] $h^{(\bullet,\bullet)}_{\omega_t}(X\,;\,\mathbb{R})$ jump varying $t$;
\item[3.] if $t\neq 0$ we have the following decomposition
$$
H^{(0,3)}_{\omega_t}(X\,;\mathbb{R})\oplus
H^{(1,1)}_{\omega_t}(X\,;\mathbb{R})\,=\,
H^3_{dR}(X\,;\,\mathbb{R})\,.
$$
However, for $t=0$ we have
$$
H^{(0,3)}_{\omega}(X\,;\mathbb{R})+
H^{(1,1)}_{\omega}(X\,;\mathbb{R})\subsetneq
H^3_{dR}(X\,;\,\mathbb{R})
$$
and
$$
H^{(0,3)}_{\omega}(X\,;\mathbb{R})\cap
H^{(1,1)}_{\omega}(X\,;\mathbb{R})\neq \left\lbrace 0\right\rbrace\,.
$$
\end{itemize}
\end{prop}

\begin{proof}
$1.$ According to
\cite{angella-kasuya-symplectic} we can compute by using invariant forms the symplectic cohomology groups of $(X\,,\omega)$
and 
they turn out to be $h^1_{d+d^\Lambda}=3$,
$h^2_{d+d^\Lambda}=8$ and $h^3_{d+d^\Lambda}=13$.
More explicitly, the second symplectic Bott-Chern cohomology group is
$$
H^2_{d+d^\Lambda}(X)=\left\langle e^{12}, e^{13}, e^{14}, e^{15}, e^{23}, e^{24}, e^{16}+e^{25}-e^{34}, e^{26}-e^{45}\right\rangle\,.
$$
Now we consider
the deformation of the symplectic structure given by
$$
\omega_t=\omega+t\left(e^{26}-e^{45}\right)=
E^{12}_t+E^{34}_t+E^{56}_t\,,
$$
where we set
$$
\left\lbrace
\begin{array}{lcl}
E^1_t & =& e^1+te^2\\
E^2_t &=& e^6\\
E^3_t &=& e^2-te^4\\
E^4_t &=& e^5\\
E^5_t &=& -e^3\\
E^6_t &=& e^4 
\end{array}
\right.\,.
$$
We consider, for any $t$, the $\omega_t$-calibrated almost-complex structure $J_t$ such that
$J_t^*E^1_t=-E^2_t$, $J_t^*E^3_t=-E^4_t$, $J_t^*E^5_t=-E^6_t$; then the
structure equations are
$$
\left\lbrace
\begin{array}{lcl}
dE^1_t & =& 0\\
dE^2_t &=& E^{14}_t-tE^{34}_t+t^2E^{46}_t-E^{35}_t+tE^{56}_t+E^{36}_t\\
dE^3_t &=& -tE^{13}_t-t^2E^{16}_t\\
dE^4_t &=& E^{16}_t-tE^{36}_t\\
dE^5_t &=& 0\\
dE^6_t &=& E^{13}_t+tE^{16}_t
\end{array}
\right..
$$
We want to show that for $t\neq 0$ $h^2_{d+d^{\Lambda_t}}<8$ so
we compute the $\Delta_{d+d^{\Lambda_t}}$-harmonic $2$-forms.
An arbitrary left-invariant $2$-form $\alpha$ on $X$ can be written as
$$
\alpha=A_{12}E^{12}_t+A_{13}E^{13}_t+A_{14}E^{14}_t+
A_{15}E^{15}_t+A_{16}E^{16}_t+
A_{23}E^{23}_t+A_{24}E^{24}_t+A_{25}E^{25}_t+A_{26}E^{26}_t+
$$
$$
+A_{34}E^{34}_t+A_{35}E^{35}_t+A_{36}E^{36}_t+
A_{45}E^{45}_t+A_{46}E^{46}_t+A_{56}E^{56}_t
$$
and by a straightforward computation using the structure
equations we get that $\alpha$ is $d$-closed and
$d^{\Lambda_t}$-closed if and only if
$$
\alpha=A_{35}\left(tE^{12}_t+tE^{34}_t+E^{35}_t\right)+A_{56}\left(E^{12}_t+E^{34}_t+E^{56}_t\right)+A_{13}E^{13}_t+
A_{14}\left(E^{14}_t-tE^{34}_t+t^2E^{46}_t\right)+
$$
$$
+A_{15}E^{15}_t+A_{16}E^{16}_t+A_{23}\left(E^{23}_t+tE^{26}_t-E^{46}_t\right)+A_{36}E^{36}_t.
$$
Such a form $\alpha$ satisfies $\left(dd^{\Lambda_t}\right)^*\alpha=0$ if
and only if
$$
\alpha=A_{35}\left(tE^{12}_t+tE^{34}_t+E^{35}_t\right)+A_{56}\left(E^{12}_t+E^{34}_t+E^{56}_t\right)+
A_{14}\left(E^{14}_t-tE^{34}_t+t^2E^{46}_t\right)+
$$
$$
+A_{15}E^{15}_t+A_{23}\left(E^{23}_t+tE^{26}_t-E^{46}_t\right)+A_{36}\left(-t^3E^{13}_t+tE^{16}_t+E^{36}_t\right).
$$
In particular by \cite{angella-kasuya-symplectic} and \cite{TY} this computation suffices to conclude that for $t\neq 0$ $h^2_{d+d^{\Lambda_t}}=6$.

$2.$ Now we want to show that also
the numbers 
$$
h^{(r,s)}_\omega\,:=\,\dim H^{(r,s)}_\omega(X\,;\,\mathbb{R})
$$
are not invariant along a symplectic deformation.

Notice that by \cite[Proposition 3.3]{angella-tomassini-symplectic} all the cohomology groups can be computed by invariant forms.
By \cite[Proposition 2.8]{angella-tomassini-symplectic} we have that 
$$
H^{(1,1)}_{\omega_t}(X\,;\,\mathbb{R})
=LH^{(0,1)}_{\omega_t}(X\,;\,\mathbb{R})=LH^1_{dR}(X\,;\,\mathbb{R})\,.
$$
Let $\alpha=Ae^1+Be^2+Ce^3$ be a $d$-closed $1$-form, with $A,B,C\in\mathbb{R}$; then
$$
\omega_t\wedge\alpha=A\left(e^{125}-e^{134}+te^{126}-te^{145}\right)+
B\left(-e^{126}-e^{234}-te^{245}\right)+
$$
$$
+C(-e^{136}-e^{235}-te^{236}
-te^{345}).
$$
Since $e^{125}$ and $e^{134}$ are $d$-exact,
one gets that for $t=0$
$$
h^{(1,1)}_\omega=2
$$
but, for $t\neq 0$
$$
h^{(1,1)}_{\omega_t}=3.
$$
More precisely, for $t\neq 0$
$$
H^{(1,1)}_{\omega_t}(X\,;\,\mathbb{R})=\mathbb{R}\left\langle
e^{126}-e^{145},e^{126}+e^{234}+te^{245},e^{136}+e^{235}+t(e^{236}+e^{345})
\right\rangle
$$
and for $t=0$,
$$
H^{(1,1)}_{\omega}(X\,;\,\mathbb{R})=\mathbb{R}\left\langle
e^{126}+e^{234},e^{136}+e^{235}
\right\rangle\,.
$$
Similarly, one can compute $H^{(0,3)}_{\omega_t}(X\,;\mathbb{R})$ obtaining
$h^{(0,3)}_{\omega_t}=3$ for $t\neq 0$ and\\
$h^{(0,3)}_{\omega}=4$.
More precisely, by linear computations one has, for $t\neq 0$
$$
H^{(0,3)}_{\omega_t}(X\,;\mathbb{R})\,=\,\mathbb{R}\left\langle
-(e^{136}+e^{235})+t(e^{345}+e^{236})-\frac{2}{3}t(-e^{156}+
e^{345}+2e^{236}+e^{246}),\right.
$$
$$
\left.
(e^{126}+e^{234})-te^{245},(e^{135}+e^{145})+te^{245}
\right\rangle
$$
and
$$
H^{(0,3)}_{\omega}(X\,;\mathbb{R})\,=\,\mathbb{R}\left\langle
e^{136}+e^{235},-e^{156}+
e^{345}+2e^{236}+e^{246},
e^{126}+e^{234},e^{135}+e^{145}\right\rangle\,.
$$

$3.$ By the above computations one gets that, for $t=0$ we have
$$
H^{(0,3)}_{\omega}(X\,;\mathbb{R})+
H^{(1,1)}_{\omega}(X\,;\mathbb{R})\subsetneq
H^3_{dR}(X\,;\,\mathbb{R})
$$
and
$$
H^{(0,3)}_{\omega}(X\,;\mathbb{R})\cap
H^{(1,1)}_{\omega}(X\,;\mathbb{R})\neq \left\lbrace 0\right\rbrace\,.
$$
But, if $t\neq 0$ we have the following decomposition
$$
H^{(0,3)}_{\omega_t}(X\,;\mathbb{R})\oplus
H^{(1,1)}_{\omega_t}(X\,;\mathbb{R})\,=\,
H^3_{dR}(X\,;\,\mathbb{R})\,.
$$
\end{proof}

\begin{cor}
Let $(X,\omega)$ be a compact symplectic manifold not satisfying the Hard-Lefschetz Condition. Then the decomposition 
$$
H^\bullet_{dR}(X\,;\,\mathbb{R})=\bigoplus_{r\in\mathbb{N}}
H^{(r,\bullet-2r)}_\omega(X\,;\,\mathbb{R})
$$
is not closed under symplectic deformations.
\end{cor}

\section{Non-HLC degrees in dimension $4$}

In this Section we give a more explicit description of \cite[Theorem 4.5]{tardini-tomassini-symplectic}. As already mentioned in the Preliminaries, in \cite{angella-tomassini-algebraic} D. Angella and the second-named author, inspired by the analogue results in complex geometry,
introduce and study on a compact symplectic manifold $(X^{2n},\omega)$
the following non-negative integers (cf. \cite{tardini-tomassini-symplectic})
\[
\Delta^k_s:= h^k_{d+d^\Lambda}-b_k,
\qquad k\in\mathbb{Z},
\]
proving that their triviality characterizes the Hard-Lefschetz Condition.
In this sense these numbers measure the HLC-degree of a symplectic manifold,
as their analogue in the complex case do (cf. \cite{angella-tomassini-inventiones}).\\
Since by \cite[Theorem 4.3]{tardini-tomassini-symplectic} the number $\Delta^1_s$ is $0$ on every compact symplectic manifold then, in dimension $4$, the vanishing of $\Delta^2_s$ characterizes the Hard-Lefschetz Condition.
We will describe this more precisely showing that $\Delta^2_s$ is the dimension of a vector subspace of the space of de Rham harmonic forms.
Recall that an almost-complex structure $J$ on a smooth manifold $X^{2n}$ acts in a natural way on the space of $k$-forms as
$$
J\alpha(v_1,\cdots,v_k):=\alpha(Jv_1,\cdots,Jv_k),
$$
with $v_1,\cdots,v_k$ vector fields on $X$. In particular, on $2$-forms $J$ acts as an involution, namely if $\alpha\in\Omega^2(X)$ then
$$
J\alpha(\cdot,\cdot\cdot)=\alpha(J\cdot,J\cdot\cdot)\quad\text{and}\quad
J^2=\text{Id},
$$
leading to a decomposition of $\Omega^2(X)$ in the $1-$ and $(-1)$-eigenspaces of $J$:
$$
\Omega^2(X)=\Omega_J^+(X)\oplus\Omega_J^-(X),
$$
where $\Omega^{\pm}_J(X):=\left\lbrace\alpha\in\Omega^2(X)\,\mid\,J\alpha=\pm\alpha\right\rbrace$,
as well as a splitting of the corresponding vector bundles
$$
\Lambda^2(X)=\Lambda_J^+(X)\oplus\Lambda_J^-(X)\,.
$$
The elements in $\Omega_J^+(X)$ (resp. $\Omega_J^-(X)$) are called \emph{$J$-invariant forms} (resp. \emph{$J$-anti-invariant forms}).\\
In \cite{LZ} T.-J. Li and W. Zhang studied when this decomposition passes in cohomology, in particular they define the \emph{$J$-invariant} and
\emph{$J$-anti-invariant subgroups} of the second de Rham cohomology group of $X$, $H^2_{dR}(X,\mathbb{R})$, taking the classes which admit a representative in $\Omega_J^+(X)$, respectively in $\Omega_J^-(X)$, i.e.,
$$
H_J^+(X)_{\mathbb{R}}:=\left\lbrace[\alpha]\in H^2_{dR}(X,\mathbb{R})\,\mid\,
J\alpha=\alpha\right\rbrace,
$$
$$
H_J^-(X)_{\mathbb{R}}:=\left\lbrace[\alpha]\in H^2_{dR}(X,\mathbb{R})\,\mid\,
J\alpha=-\alpha\right\rbrace.
$$
Clearly, we have the inclusion
$$
H_J^+(X)_{\mathbb{R}}+H_J^-(X)_{\mathbb{R}}\subseteq
H^2_{dR}(X,\mathbb{R})
$$
but the sum can be neither direct nor equal to $H^2_{dR}(X,\mathbb{R})$.
Therefore we recall the following definition
\begin{definition}[{\cite[Definitions 2.2, 2.3, Lemma 2.2]{LZ}}]
An almost-complex structure $J$ on a differentiable manifold $X$ is called
\begin{itemize}
\item[-] \emph{$\mathcal{C}^\infty$-pure} if 
\[
H_J^+(X)_{\mathbb{R}}\cap H_J^-(X)_{\mathbb{R}}=\left\lbrace 0\right\rbrace,
\]
\item[-]  \emph{$\mathcal{C}^\infty$-full} if
\[
H^2_{dR}(X;\mathbb{R})= H_J^+(X)_{\mathbb{R}}+H_J^-(X)_{\mathbb{R}},
\]
\item[-] \emph{$\mathcal{C}^\infty$-pure and full} if it is $\mathcal{C}^\infty$-pure and
$\mathcal{C}^\infty$-full, i.e., if
\[
H^2_{dR}(X;\mathbb{R})=H_J^+(X)_{\mathbb{R}}\oplus H_J^-(X)_{\mathbb{R}}.
\]
\end{itemize}
\end{definition}
Examples of $\mathcal{C}^\infty$-pure and full structures are K\"ahler structures.
More in general, Draghici, Li and Zhang prove that any
almost-complex structure on a compact $4$-manifold is $\mathcal{C}^\infty$-pure and full (see \cite[Theorem 2.3]{DLZ}). A similar result does not hold in higher dimensions, indeed in \cite{FT} a $6$-dimensional non-$\mathcal{C}^{\infty}$-pure almost-complex manifold is constructed.\\
Let $(X^4,\omega)$ be a compact symplectic manifold of dimension $4$ and let $J$ be a compatible almost-complex structure and $g$ the associated Riemannian metric, then also the Hodge operator $*_g$ associated to $g$ acts as an involution on $\Lambda^2(X)$ giving the splittings
$$
\Lambda^2(X)=\Lambda_g^+\oplus\Lambda_g^-\,,\qquad
\Omega^2(X)=\Omega_g^+\oplus\Omega_g^-
$$
into \emph{selfdual} and \emph{anti-selfdual} $2$-forms. Moreover,
$$
\mathcal{H}^2_{dR}(X)=\mathcal{H}^+_{g}\oplus\mathcal{H}^-_{g}
$$
where $\mathcal{H}^2_{dR}(X)$ denotes the space of de Rham harmonic forms and $\mathcal{H}^{\pm}_{g}:=\Lambda_g^{\pm}\cap \Ker d$.
The two splittings are related as follows
$$
\Lambda_g^+=\mathbb{R}\omega\oplus\Lambda_J^-\,,\qquad
\Lambda_J^+=\mathbb{R}\omega\oplus\Lambda_g^-\,.
$$
Now we prove the following
\begin{theorem}\label{thm:dr-included-bc-new}
Let $(X^{2n}, J, g, \omega)$ be a compact almost-K\"ahler manifold and suppose that $J$ is $\mathcal{C}^\infty$-pure and full. Then
$$
\mathcal{H}^2_{dR}(X)\subseteq \mathcal{H}^2_{d+d^\Lambda}(X)\,.
$$
In particular, $b_2(X)\leq h^2_{d+d^\Lambda}(X)$.
\end{theorem}
Recall that by \cite[Proposition 3.2]{FT} on a compact almost-K\"ahler manifold $J$ is always $\mathcal{C}^\infty$-pure.\\
\begin{proof}
Let $\alpha\in\mathcal{H}^2_{dR}(X)$, namely $d\alpha=0$ and $d*\alpha=0$. We just need to show that $d^\Lambda\alpha=0$ since, clearly
$dd^\Lambda*\alpha=-d^\Lambda d*\alpha=0$.
Notice that $d^\Lambda\alpha=0$ if and only if $d*J\alpha=0$.
We consider the unique decomposition $\alpha=\gamma_J^++\gamma_J^-$
where $\gamma_J^+$ is a $J$-invariant form and $\gamma_J^-$ is a $J$-anti-invariant form, hence $J\alpha=\gamma_J^+-\gamma_J^-$.
Therefore, the thesis is equivalent to show that
$$
d*\gamma_J^+=d*\gamma_J^-\,.
$$
By hypothesis,
$$
0=d\alpha=d\gamma_J^++d\gamma_J^-
$$
and
$$
0=d*\alpha=d*\gamma_J^++d*\gamma_J^-
$$
so, we have to prove that $d*\gamma_J^-=0$.
By \cite[Proposition 3.1]{HMT} one has that
$$
d*\gamma_J^-=\frac{1}{(n-2)!}\,d\gamma_J^-\,\wedge\,\omega^{n-2}\,,
$$
hence, if we show that $d\gamma_J^-=0$ we can conclude.\\
By hypothesis, $H^2_{dR}(X)=H_J^+(X)\oplus H_J^-(X)$ and, recall that $Z_J^-(X)\to H_J^-(X)$ is bijective namely, we have a decomposition, with the obvious notations,
$$
\alpha=\alpha_J^++(d\beta)_J^++\alpha_J^-\,.
$$
By uniqueness, $\gamma_J^-=\alpha_J^-$ and now,
$d\gamma_J^-=d\alpha_J^-=0$ concluding the proof.
\end{proof}
Since by \cite{DLZ} on a compact smooth $4$-manifold every almost-complex structure is $\mathcal{C}^\infty$-pure and full one gets the following
\begin{cor}
Let $(X^4,\omega)$ be a compact symplectic manifold. Then,
$$
\mathcal{H}^2_{dR}(X)\subseteq \mathcal{H}^2_{d+d^\Lambda}(X)\,.
$$
In particular, $b_2(X)\leq h^2_{d+d^\Lambda}(X)$.
\end{cor}
\begin{proof}
Fix a $\omega$-compatible almost-complex structure $J$ on $X$. Then, by \cite{DLZ} $J$ is $\mathcal{C}^\infty$-pure and full hence we can apply Theorem \ref{thm:dr-included-bc-new} to $(X^4,J,\omega)$.
\end{proof}
\begin{rem}
By definition $\Delta^2_s:=h^2_{d+d^\Lambda}-b_2$ and so on a compact almost-K\"ahler manifold with $J$ $\mathcal{C}^\infty$-pure and full we just proved, with a different technique, that $\Delta_s^2\geq 0$.\\
In particular, if $(X^4,\omega)$ is a compact symplectic manifold one has that
$\Delta_s^2\geq 0$ and, moreover
$$
\mathcal{H}^2_{d+d^\lambda}(X)=\mathcal{H}^2_{dR}(X)
$$
if and only if
$$
b_2= h^2_{d+d^\Lambda}
$$
which is equivalent to the Hard-Lefschetz Condition by \cite{tardini-tomassini-symplectic}.
\end{rem}
Let $(X^{2n}, J, g, \omega)$ be a compact almost-K\"ahler manifold and suppose that $J$ is $\mathcal{C}^\infty$-pure and full.\\
In view of Theorem \ref{thm:dr-included-bc-new} we denote with $V$ the finite dimensional vector subspace of $\mathcal{H}^2_{dR}(X)$ such that
$$
\mathcal{H}^2_{dR}(X)\oplus V= \mathcal{H}^2_{d+d^\lambda}(X)\,.
$$
It follows that
$$
\dim V=h^2_{d+d^\Lambda}-b_2=\Delta^2_s
$$
namely, $\Delta^2_s$ can be seen as the dimension of a vector subspace of the space of harmonic forms $\mathcal{H}^2_{dR}(X)$.
Before discussing who $V$ is, we see a couple of examples.

\begin{ex}
Let $X=\Gamma\backslash G$ be a primary-Kodaira surface with structure equations
\[
\left\{\begin{array}{rcl}
            d e^1 &=&   0 \\[5pt]
            d e^2 &=&   0  \\[5pt]
            d e^3 &=&   0  \\[5pt]
            d e^4 &=&   e^2\wedge e ^3  \\[5pt]
       \end{array}\right. \;.
\]
Let $\omega:=e^1\wedge e^2+e^3\wedge e^4$ be a symplectic structure and set $Je^1=-e^2$, $Je^3=-e^4$. We denote with $g$ the associated Riemannian metric. Then, with the usual notation $e^{ij}:=e^i\wedge e^j$, it is easy to show that
$$
\mathcal{H}^2_{dR}(X)=\left\langle e^{12},e^{13},e^{24},e^{34}\right\rangle=
\left\langle e^{12}+e^{34},e^{13}+e^{24}\right\rangle\oplus
\left\langle e^{12}-e^{34},e^{13}-e^{24}\right\rangle
$$
with
$$
\mathcal{H}^+_{g}=\left\langle e^{12}+e^{34},e^{13}+e^{24}\right\rangle\,,
$$
and
$$
\mathcal{H}^-_{g}=\left\langle e^{12}-e^{34},e^{13}-e^{24}\right\rangle\,.
$$
Similarly, one can compute
$$
\mathcal{H}^2_{d+d^\Lambda}(X)=\left\langle e^{12},e^{13},e^{24},e^{34},e^{23}\right\rangle\,,
$$
in particular
$$
V=\left\langle e^{23}\right\rangle\,.
$$
Notice that $e^{23}$ is $d$-exact.
\end{ex}
\begin{ex}
Let $X=\Gamma\backslash G$ be a surface with structure equations
\[
\left\{\begin{array}{rcl}
            d e^1 &=&   0 \\[5pt]
            d e^2 &=&   0  \\[5pt]
            d e^3 &=&   e^1\wedge e^2  \\[5pt]
            d e^4 &=&   e^1\wedge e ^3  \\[5pt]
       \end{array}\right. \;.
\]
Let $\omega:=e^1\wedge e^4+e^2\wedge e^3$ be a symplectic structure and set $Je^1=-e^4$, $Je^2=-e^3$. We denote with $g$ the associated Riemannian metric. Then, with the usual notation $e^{ij}:=e^i\wedge e^j$, it is easy to show that
$$
\mathcal{H}^2_{dR}(X)=\left\langle e^{14},e^{23}\right\rangle=
\left\langle e^{14}+e^{23}\right\rangle\oplus
\left\langle e^{14}-e^{23}\right\rangle
$$
with
$$
\mathcal{H}^+_{g}=\left\langle e^{14}+e^{23}\right\rangle\,,
$$
and
$$
\mathcal{H}^-_{g}=\left\langle e^{14}-e^{23}\right\rangle\,.
$$
Similarly, one can compute
$$
\mathcal{H}^2_{d+d^\Lambda}(X)=\left\langle e^{12},e^{13},e^{14},e^{23}\right\rangle\,,
$$
in particular
$$
V=\left\langle e^{12},e^{13}\right\rangle\,.
$$
Notice that $e^{12}$ and $e^{13}$ are $d$-exact.
\end{ex}
Now we describe explicitly the space $V$.
\begin{prop}
Let $(X^{2n},\omega,J,g)$ be a compact almost-K\"ahler manifold with $J$ $\mathcal{C}^\infty$-pure and full. Then,
$$
V=\mathcal{H}^2_{d+d^\Lambda}(X)\cap \Imm d\,.
$$
\end{prop}
\begin{proof}
By definition of $V$
$$
\mathcal{H}^2_{dR}(X)\oplus V= \mathcal{H}^2_{d+d^\lambda}(X)\,,
$$
and by Hodge theory one has that
$$
\mathcal{H}^2_{d+d^\lambda}(X)\subset
\Lambda^2(X)\cap\Ker d=\mathcal{H}^2_{dR}(X)\oplus d\left(A^1(X)\right)\,.
$$
Hence, one gets
$$
V=\mathcal{H}^2_{d+d^\Lambda}(X)\cap d\left(A^1(X)\right)
$$
concluding the proof.
\end{proof}
Putting together this and \cite[Theorem 4.3]{tardini-tomassini-symplectic} one has the following
\begin{cor}\label{cor:delta2-dimensionofv}
Let $(X^4,\omega)$ be a compact symplectic $4$-manifold and let $g$ be a compatible Riemannian metric. Then, the non-HLC degrees are
$$
\Delta_s^1=0 \quad\text{and}\quad
\Delta_s^2=\dim\left(\mathcal{H}^2_{d+d^\Lambda}(X)\cap \Imm d\right)\,.
$$
Moreover,
$\Delta_s^2=0$ if and only if $V=\left\lbrace 0\right\rbrace$ if and only if the Hard-Lefschetz Condition holds on $(X^4\,,\omega)$.
\end{cor}
We show now an explicit example of an almost-K\"ahler $6$-dimensional manifold $X$ with $J$ non $\mathcal{C}^\infty$-full but such that
$\mathcal{H}^2_{dR}(X)\subseteq \mathcal{H}^2_{d+d^\lambda}(X)$. Moreover,
$\mathcal{H}^3_{dR}(X)\not\subseteq\mathcal{H}^3_{d+d^\lambda}(X)$.
\begin{ex}
Let $X=\Gamma\backslash G$ be the $6$-dimensional nilmanifold with structure equations
\[
\left\{\begin{array}{rcl}
            d e^1 &=&   0 \\[5pt]
            d e^2 &=&   e^{16}+e^{35}  \\[5pt]
            d e^3 &=&   0  \\[5pt]
            d e^4 &=&   e^{15}-e^{36}\\[5pt]
            d e^5 &=&   0  \\[5pt]
            d e^6 &=&   0  \\[5pt]
       \end{array}\right. \;.
\]
Notice that this is the real manifold underlying the Iwasawa manifold.
Let $\omega=e^{12}+e^{34}+e^{56}$ be a symplectic structure and let $J$ be the compatible almost-complex structure such that $Je^1=-e^2$,
$Je^3=-e^4$ and $Je^5=-e^6$. We denote with $g$ the associated Riemannian metric. The space of de-Rham harmonic forms is
$$
\mathcal{H}_{dR}^2(X)=\mathbb{R}\left\langle
e^{12}+e^{34},e^{13},e^{14}+e^{23},e^{15}+e^{36},e^{16}-e^{35},
e^{25}+e^{46},e^{26}-e^{45},e^{56}
\right\rangle\,.
$$
Notice that $J$ is not $\mathcal{C}^\infty$-full (cf. \cite{ATZ}), indeed
$$
H_J^+(X)=\mathbb{R}\left\langle
[e^{12}+e^{34}],[e^{15}+e^{36}+e^{26}-e^{45}],[e^{16}-e^{35}-e^{25}-e^{46}],[e^{56}]
\right\rangle\,
$$
and
$$
H_J^-(X)=\mathbb{R}\left\langle
[e^{14}+e^{23}],[e^{15}+e^{36}-e^{26}+e^{45}],[e^{16}-e^{35}+e^{25}+e^{46}]
\right\rangle\,
$$
hence $H_J^+(X)\oplus H_J^-(X)\subsetneq H^2_{dR}(X)$.\\
However, thanks to \cite{macri} and \cite{angella-kasuya-symplectic}
we can also compute by using invariant forms the symplectic Bott-Chern cohomology and one has
$$
\mathcal{H}_{d+d^\Lambda}^2(X)=\mathbb{R}\left\langle
e^{12}+e^{34},e^{13},e^{14}+e^{23},e^{15},e^{16},
e^{25}+e^{46},e^{26}-e^{45},e^{35},e^{36},e^{56}
\right\rangle\,,
$$
so $\mathcal{H}_{dR}^2(X)\subseteq \mathcal{H}_{d+d^\Lambda}^2(X)$. Of course, in this case
$$
V=\mathbb{R}\left\langle
e^{15}-e^{36},e^{16}+e^{35}
\right\rangle=\mathcal{H}_{d+d^\Lambda}^2(X)\cap \Imm d.
$$
Moreover, notice that
$\mathcal{H}_{dR}^3(X)\not\subseteq \mathcal{H}_{d+d^\Lambda}^3(X)$,
indeed if we consider the $3$-form $e^{256}$, by a straightforward computation one has $de^{256}=0$ and $d*e^{256}=0$, but
$d^\Lambda e^{256}=-e^{16}-e^{35}\neq 0$.
\end{ex}
\begin{rem}
Recall that the $6$-dimensional nilmanifolds can be classified in terms of their Lie algebra, up to isomorphisms, in $34$ classes, according to V. V. Morozovs classification \cite{morozov} and $26$ of them admit a symplectic structure.
With explicit computations one can show that all the symplectic $6$-dimensional nilmanifolds with left-invariant symplectic structure as listed for instance in Table $3$ in \cite{angella-kasuya-symplectic} satisfy $\mathcal{H}_{dR}^2(X)\subseteq \mathcal{H}_{d+d^\Lambda}^2(X)$. We ask whether on nilmanifolds endowed with a left-invariant symplectic structure it holds  $\mathcal{H}_{dR}^2(X)\subseteq \mathcal{H}_{d+d^\Lambda}^2(X)$.
\end{rem}

\end{document}